\documentclass[12pt]{amsart}
\usepackage{amssymb}
\usepackage{enumerate}
\usepackage{graphicx}
\makeatletter
\@namedef{subjclassname@2010}{%
  \textup{2010} Mathematics Subject Classification}
\makeatother
\newtheorem{thm}{Theorem}[section]
\newtheorem{corollary}[thm]{Corollary}
\newtheorem{lemma}[thm]{Lemma}
\newtheorem{proposition}[thm]{Proposition}
\theoremstyle{definition}
\newtheorem{definition}[thm]{Definition}
\newtheorem{remark}[thm]{Remark}
\newtheorem{example}[thm]{Example}
\begin{document}
\baselineskip=15pt
\title{On $1$-absorbing prime and weakly $1$-absorbing prime subsemimodules}

\author[M. Adarbeh]{Mohammad Adarbeh $^{(\star)}$}
\address{Department of Mathematics, Birzeit University, Birzeit,  Palestine}
\email{madarbeh@birzeit.edu}
\author[M. Saleh]{Mohammad Saleh  }
\address{Department of Mathematics, Birzeit University, Birzeit,  Palestine}
\email{msaleh@birzeit.edu}

\thanks{$^{(\star)}$ Corresponding author}
\date{}

\begin{abstract}  In this paper, we introduce the concepts of $1$-absorbing prime and weakly $1$-absorbing prime subsemimodules over commutative semirings. Let $S$ be a commutative semiring with $1 \neq 0$ and $M$ an $S$-semimodule. A proper subsemimodule $N$ of $M$ is called $1$-absorbing prime (weakly $1$-absorbing prime) if, for all nonunits $a, b \in S$ and $m \in M$, $abm \in N$ ($0 \neq abm \in N$) implies $ab \in (N :_{S} M)$ or $m \in N$. We study many properties of these concepts. For example, we show that a proper subsemimodule $N$ of $M$ is $1$-absorbing prime if and only if for all proper ideals $I, J$ of $S$ and subsemimodule $K$ of $M$ with $IJK \subseteq N$, either $IJ \subseteq (N:_{S} M)$ or $K \subseteq N$. Also, we prove that a proper subtractive subsemimodule $N$ of $M$ is weakly $1$-absorbing prime if and only if for all proper ideals $I, J$ of $S$ and subsemimodule $K$ of $M$ with $0 \neq IJK \subseteq N$, either $IJ \subseteq (N:_{S} M)$ or $K \subseteq N$.
\end{abstract}



\maketitle

\section{Introduction}

\noindent
Semirings are important algebraic structures with numerous applications in science and engineering. The concept of a semiring was first introduced by Vandiver \cite{V} in 1934. Throughout, $S$ denotes a commutative semiring with a nonzero identity and $M$ a unital $S$-semimodule. Let $N$ be a subsemimodule of $M$. The set $\{s\in S\mid sM\subseteq N\}$ forms an ideal of $S$ and is denoted by $(N:_{S} M)$. In particular, the ideal $(0:_{S} M)$ is called the annihilator of $M$ and is denoted by $\text{Ann}(M)$. If $\text{Ann}(M)=0$, then $M$ is called a faithful $S$-semimodule \cite{NG}. $N$ is called subtractive if, whenever $x,x+y\in N$, then $y\in N$ for all $x,y\in M$. A subset $T$ of $S$ is called multiplicatively closed if $1\in T$ and $xy \in T$ for all $x,y \in T$. Let $T$ be a multiplicatively closed subset $S$. The localization of $M$ at $T$ is defined as follows: First, define an equivalence relation $\sim$ on $M\times T$ by $(m, t) \sim (m',t')$ if and only if $stm' =st'm$ for some $s\in T$. Let $\frac{m}{t}$ denote the equivalence class of $(m,t)\in M\times T$ and let $T^{-1}M$ denote the set of all equivalence classes of $M\times T$. Define addition on $T^{-1} M$ by $\frac{m}{t} +\frac{m'}{t'} = \frac{t'm+tm'}{tt'}$. For $\frac{s}{t}\in T^{-1}S$ and $\frac{m}{t'}\in T^{-1}M$, define $\frac{s}{t}\cdot\frac{m}{t'}= \frac{sm}{tt'}$. Then $T^{-1}M$ is an $T^{-1}S$-semimodule \cite{D}. For more details on semirings, we may refer to \cite{REA,G,PNS,V}. 

Let $R$ be a commutative ring with $1\not=0$. A proper ideal $I$ of $R$ is called prime if for all $a, b \in R$, $ab \in I$ implies $a \in I$ or $b \in I$. Anderson and Smith \cite{AS} introduced the concept of weakly prime ideals of commutative rings as a generalization of the concept of prime ideals. A proper ideal $I$ of $R$ is called weakly prime if for all $a, b \in R$, $0 \not= ab \in I$ implies $a \in I$ or $b \in I$. Badawi \cite{Ab_{2}} introduced the notion of $2$-absorbing ideals of commutative rings as a generalization of the notion of prime ideals. A proper ideal $I$ of $R$ is called $2$-absorbing if for all $a,b,c\in R$, $abc \in I$ implies $ab \in I$ or $ac \in I$ or $bc \in I$. Badawi and Youseﬁan \cite{AB3} introduced the notion of weakly $2$-absorbing ideals of commutative rings as a generalization of the notion of weakly prime ideals. A proper ideal $I$ of $R$ is called weakly $2$-absorbing if for all $a,b,c\in R$, $0\not=abc \in I$ implies $ab \in I$ or $ac \in I$ or $bc \in I$. Yassine et al. \cite{NNY} introduced the concept of $1$-absorbing prime ideals of commutative rings, which is weaker than the concept of prime ideals and stronger than the concept of $2$-absorbing ideals. A proper ideal $I$ of $R$ is called $1$-absorbing prime if for all nonunits $a,b,c\in R$, $abc \in I$ implies $ab \in I$ or $c \in I$. Koç et al. \cite{KTY} introduced the concept of weakly $1$-absorbing prime ideals of commutative rings, which is weaker than the concept of weakly prime ideals and stronger than the concept of weakly $2$-absorbing ideals. A proper ideal $I$ of $R$ is called weakly $1$-absorbing prime if for all nonunits $a,b,c\in R$, $0\not=abc \in I$ implies $ab \in I$ or $c \in I$. Recently, Ugurlu \cite{U} introduced the concept of $1$-absorbing prime submodules over commutative rings. A proper submodule $N$ of an $R$-module $M$ is called $1$-absorbing prime if for all nonunits $a, b \in R$ and $m \in M$, $abm \in N$ implies $ab \in (N :_{R} M)$ or $m \in N$. For more details on the absorbing like-properties in commutative (semi)rings, we may refer to \cite{MAMS,Ab_{2},AB3,BY,BYC,BN,KTY,PB,MSIM,NNY}.

In this paper, we define and study the concepts of $1$-absorbing prime and weakly $1$-absorbing prime subsemimodules over commutative semirings. Let $S$ be a commutative semiring and $M$ be an $S$-semimodule. A proper subsemimodule $N$ of $M$ is called $1$-absorbing prime (weakly $1$-absorbing prime) if for all nonunits $a, b \in S$ and $m \in M$, $abm \in N$ ($0\not = abm\in N$) implies $ab \in (N :_{S} M)$ or $m\in N$.      

In Section 2, we investigate some properties of $1$-absorbing prime subsemimodule. For example, we prove that if $N$ is a $1$-absorbing prime subsemimodule of $M$, then $(N:_{S} M)$ is a $1$-absorbing prime ideal of $S$. The converse of the last fact is not true in general (see Example \ref{ex2}), however; it is true if $M$ is an $MC$ multiplication $S$-semimodule (see Corollary \ref{cormc} (1)). In Theorem \ref{thmlocal}, we show that if an $S$-semimodule $M$ has a $1$-absorbing prime subsemimodule that is not prime, then $S$ is a local semiring. In Theorem \ref{th1}, we give some characterizations of $1$-absorbing prime subsemimodules. Theorem \ref{th2N} characterizes $1$-absorbing prime subsemimodules of multiplication $S$-semimodules. More precisely, we prove that a proper subsemimodule $N$ of a multiplication $S$-semimodule $M$ is $1$-absorbing prime if and only if for all subsemimodules $N_1,N_2,N_3$ of $M$ with $N_1N_2N_3\subseteq N$, either $N_1N_2 \subseteq N$ or $N_3 \subseteq N$. The last fact proves that the condition "$M$ is a faithful finitely generated module" in \cite[Theorem 4.1]{U} can be removed (see Corollary \ref{corg}). 

In Section 3, we define the concept of weakly $1$-absorbing prime subsemimodule as a generalization of the concept of $1$-absorbing prime subsemimodule. After defining this concept, we study some of its properties. Firstly, we will provide an example (Example \ref{exwnot1}) of a weakly $1$-absorbing prime subsemimodule that is not $1$-absorbing prime. In Theorem \ref{th3}, we give some characterizations of weakly $1$-absorbing prime subtractive subsemimodules. The last result in this section proves that if $M$ is a multiplication semimodule over a local semiring $S$ and $N$ is a weakly $1$-absorbing prime subtractive subsemimodule of $M$ that is not $1$-absorbing prime, then $N^3=0$.

\section{$1$-absorbing prime subsemimodules}\label{d}
We start this section with the following definition of $1$-absorbing prime subsemimodules over commutative semirings:
\begin{definition}\label{d1}
Let $S$ be a commutative semiring and $M$ an $S$-semimodule. A proper subsemimodule $N$ of $M$ is called $1$-absorbing prime if for all nonunits $a, b \in S$ and $m \in M$, $abm \in N$ implies $ab \in (N:_{S} M)$ or $m \in N$. 
\end{definition}

Recall that a semiring $S$ is called local if it has exactly one maximal ideal. Equivalently, the sum of two nonunits of $S$ is a nonunit of $S$ \cite{PNV}. 
\begin{example}\label{e1}
    Let $S$ be a local semiring with maximal ideal $\mathfrak{m}$ such that $\mathfrak{m}^2=0$. Let $M$ be any $S$-semimodule. Then every proper subsemimodule of $M$ is $1$-absorbing prime. To see this, let $N$ be any proper subsemimodule of $M$ and suppose that $abx \in N$ for some nonunits $a,b \in S$ and some $x\in M$. Then $a,b\in \mathfrak{m}$ and so $ab\in \mathfrak{m}^{2}=0$. Hence $ab\in (N:_{S} M)$. Thus $N$ is $1$-absorbing prime.
\end{example}

Let $S$ be a semiring and $M$ an $S$-semimodule. A proper subsemimodule $N$ of $M$ is called prime if whenever $am\in N$ for $a\in S$ and $m\in M$, either $a\in (N:_{S} M)$ or $m\in N$ \cite{E}.
     
\begin{remark}\label{p1} Let $S$ be a semiring and $M$ an $S$-semimodule. Then every prime subsemimodules $N$ of $M$ is $1$-absorbing prime.
\end{remark}
\begin{proof}
   Suppose that $abm \in N$ for some nonunits $a, b \in S$ and $m\in M$. Since $N$ is a prime subsemimodules, either $ab \in (N:_{S} M)$ or $m \in N$. Hence $N$ is $1$-absorbing prime.
\end{proof} 

\begin{thm}\label{thmlocal}
    Let $M$ be an $S$-semimodule. If $M$ has a $1$-absorbing prime subsemimodule that is not prime, then $S$ is a local semiring.
\end{thm}

\begin{proof}
    Suppose that $N$ is a $1$-absorbing prime subsemimodule that is not prime. Then there exists a nonunit $s \in S$; and $m \in M$ such that $sm \in N$ but $s\notin (N:_{S} M)$ and $m \notin N$ . Let $t_{1},t_{2}\in S$ be nonunits. Since $N$ is $1$-absorbing prime, $t_{i}sm \in N$, and
$m\notin N$, then $t_{i}s \in (N:_{S} M)$ for $i=1,2$. It follows that $(t_{1}+t_{2})s\in (N:_{S} M)$.  If $t_{1} + t_{2}$ is a unit in $S$, then $s=(t_{1} + t_{2})^{-1}(t_{1} + t_{2})s\in (N:_{S} M)$, a contradiction. So $t_{1} + t_{2}$ is a nonunit in $S$. Hence, the sum of two nonunits in $S$ is again a nonunit in $S$. Thus $S$ is a local semiring.
\end{proof}

\begin{corollary}\label{corlocal}
    Let $S$ be a non-local semiring and $M$ an $S$-semimodule. Then a proper subsemimodule $N$ of $M$ is $1$-absorbing prime if and only if it is prime.
\end{corollary}

\begin{proposition}\label{p2}
     Let $S$ be a semiring and $M$ an $S$-semimodule. If $N$ is a $1$-absorbing prime subsemimodule of $M$, then $(N:_{S} M)$ is a $1$-absorbing prime ideal of $S$.
\end{proposition}
\begin{proof}
   Assume that $abc \in (N:_{S} M)$, where $a, b, c$ are nonunits of $S$, and suppose $ab\notin (N:_{S} M)$. We prove that $c\in (N:_{S} M)$. Let $m \in M$. Since $abc \in (N:_{S} M)$, then $abcM\subseteq N$. Hence $abcm \in N$. But since $N$ is $1$-absorbing prime and $ab\notin (N:_{S} M)$, so $cm \in N$. Since $m$ was arbitrary, then $c\in (N:_{S} M)$. Therefore, $(N:_{S} M)$ is a $1$-absorbing prime ideal of $S$.
\end{proof}

Let $\mathbb{Z}^{\circ}$ denote the semiring of non-negative integers with the usual addition and multiplication. The converse of Proposition \ref{p2} is not true in general, as shown in the following example:

    \begin{example} \label{ex2}
    Let $S=\mathbb{Z}^{\circ}$, $M=\mathbb{Z}^{\circ}\times \mathbb{Z}^{\circ}$, and $N=2\mathbb{Z}^{\circ}\times 0$. Since $(N:_{S} M)=\{0\}$ is a prime ideal of $S$, then it is a $1$-absorbing prime ideal of $S$. However, $N$ is not $1$-absorbing prime subsemimodule of $M$ since $2\cdot 2\cdot (1,0)=(4,0)\in N$ but $4\notin (N:_{S} M)$ and $(1,0)\notin N$. 
 \end{example}

 Recall that if $N$ is a subsemimodule of $M$ and $J$ is a subset of $S$, then the residual of $N$ by $J$ is the set $(N:_{M} J)=\{m\in M: Jm\subseteq N\}$. The following theorem gives some characterizations of $1$-absorbing prime subsemimodules.

\begin{thm}\label{th1}
    Let $M$ be an $S$-semimodule and $N$ a proper subsemimodule of $M$. Then the following statements are equivalent:
    \begin{enumerate}
    \item[(1)] $N$ is a $1$-absorbing prime subsemimodule of $M$.
   \item[(2)] For all nonunits $a,b\in S$ with $ab \notin (N:_{S} M)$, $(N :_{M} ab) \subseteq N$.
   \item[(3)] For all nonunits $a,b\in S$ and subsemimodule $K$ of $M$ with $abK \subseteq N$, either $ab \in (N:_{S} M)$ or $K \subseteq N$.
  \item[(4)] For all proper ideals $I,J$ of $S$ and subsemimodule $K$ of $M$ with $IJK \subseteq N$, either $IJ \subseteq (N:_{S} M)$ or $K \subseteq N$.
    \end{enumerate}
\end{thm} 

\begin{proof}
 (1) $\Rightarrow$ (2). Let $a,b\in S$ be nonunits and suppose that $ab \notin (N:_{S} M)$.
Take $m \in (N :_{M} ab)$. Then $abm \in N$. But since $N$ is $1$-absorbing prime and
$ab \notin (N:_{S} M)$, we have $m \in N$. Thus $(N :_{M} ab) \subseteq N$.

(2) $\Rightarrow$ (3). Assume that $abK \subseteq N$ for some nonunits $a,b\in S$ and some subsemimodule $K$ of $M$. If $ab \notin (N:_{S} M)$, then by (2), $(N :_{M} ab)\subseteq N$. So $K \subseteq (N :_{M} ab) \subseteq N$.

(3) $\Rightarrow$ (4). Let $I,J$ be proper ideals of $S$ and $K$ be a subsemimodule of $M$ such that $IJK \subseteq N$. Suppose that $IJ \nsubseteq (N:_{S} M)$. Then there exist nonunits $a \in I$ and $b \in J$ such that $ab \notin (N:_{S} M)$. Since $abK \subseteq IJK \subseteq N$, then by (3) and since $ab \notin (N:_{S} M)$, we have $K \subseteq N$.
 
(4) $\Rightarrow$ (1). Assume that $abm \in N$, where $a,b\in S$ are nonunits and $m \in M$. Take $I=
Sa$, $J = Sb$, and $K = Sm$. Then we have $IJK=(Sa)(Sb)(Sm)=Sabm \subseteq N$. So by (4), $IJ \subseteq (N:_{S} M)$ or $K \subseteq N$. Hence $ab \in (Sa)(Sb)=IJ \subseteq (N:_{S} M)$ or $m \in Sm=K \subseteq N$. Therefore, $N$ is a $1$-absorbing prime subsemimodule of $M$.
\end{proof}	
Recall that an $S$-semimodule $M$ is called a multiplication $S$-semimodule if for each subsemimodule $N$ of $M$, there is an ideal $I$ of $S$ such that $N = IM$. If $M$ is a multiplication $S$-semimodule and $N$ is a subsemimodule of $M$, then $N = (N:_{S} M)M$ \cite{NG}. In the next theorem, we give a characterization of $1$-absorbing prime subsemimodules of multiplication $S$-semimodules.  

 \begin{thm}\label{th2N}
   Let $M$ be a multiplication $S$-semimodule and $N$ a proper subsemimodule of $M$. Then the following statements are equivalent:
   \begin{enumerate}
    \item[(1)] $N$ is a $1$-absorbing prime subsemimodule of $M$.
  \item[(2)] For all subsemimodules $N_1,N_2,N_3$ of $M$ with $N_1N_2N_3\subseteq N$, either $N_1N_2 \subseteq N$ or $N_3 \subseteq N$.
    \end{enumerate}   
 \end{thm}

 \begin{proof}
     (1) $\Rightarrow$ (2). Suppose that $N_1N_2N_3\subseteq N$ for some subsemimodules $N_1,N_2,N_3$ of $M$. Since $M$ is multiplication $S$-semimodule, then for $i=1,2,3$, $N_i=I_iM$ for some ideal $I_i$ of $S$. So $I_1I_2I_3M\subseteq N$. By Theorem \ref{th1}, $I_1I_2\subseteq (N:_S M)$ or $I_3M\subseteq N$. So $I_1I_2M\subseteq N$ or $I_3M\subseteq N$. That is, $N_1N_2 \subseteq N$ or $N_3 \subseteq N$.\\
       (2) $\Rightarrow$ (1). Suppose that $I_1I_2K\subseteq N$ for some proper ideals $I_1,I_2$ of $S$ and some subsemimodule $K$ of $M$. Since $M$ is multiplication $S$-semimodule, we have $K=I_3M$ for some ideal $I_3$ of $S$. So we have $I_1I_2I_3M\subseteq N$. Let $N_1=I_1M$ and $N_2=I_2M$. Then $N_1N_2K=I_1I_2I_3M\subseteq N$. By (2), $N_1N_2\subseteq N$ or $K\subseteq N$. But if $N_1N_2\subseteq N$, then $I_1I_2M\subseteq N$ and so $I_1I_2\subseteq (N:_S M)$. Hence $I_1I_2\subseteq (N:_S M)$ or $K\subseteq N$. Thus by Theorem \ref{th1}, $N$ is a $1$-absorbing prime subsemimodule of $M$.
 \end{proof}

 \begin{corollary}\label{corg}
     Let $R$ be a commutative ring and $M$ a multiplication $R$-module. Then $N$ is a $1$-absorbing prime submodule of $M$ if and only if for each submodules $N_1,N_2,N_3$ of $M$ with $N_1N_2N_3\subseteq N$, either $N_1N_2 \subseteq N$ or $N_3 \subseteq N$.
 \end{corollary}

 \begin{corollary}\cite[Theorem 4.1]{U}.
      Let $M$ be a faithful finitely generated multiplication $R$-module. Then $N$ is a $1$-absorbing prime submodule of $M$ if and only if for each submodules $N_1,N_2,N_3$ of $M$ with $N_1N_2N_3\subseteq N$, either $N_1N_2 \subseteq N$ or $N_3 \subseteq N$.
 \end{corollary}

 Let $M$ be an $S$-semimodule and $\mathfrak{m}$ be a maximal ideal of $S$. Recall that
 \begin{enumerate}
     \item[(i)] $M$ is called multiplicatively cancellative (abbreviated as $MC$) if for any $s,s' \in S$ and $0\not= x \in M$, $sx=s'x$ implies $s = s'$ \cite{NG}. 
 \item[(ii)] $M$ is called $\mathfrak{m}$-cyclic if there exist $s\in S$, $q \in \mathfrak{m}$, and $x\in M$ such that $s+q=1$ and $sM \subseteq Sx$ \cite{NG}.  
 \end{enumerate}

\begin{lemma} \cite[Theorem 4]{NG} \label{lemjm}
    Suppose that $M$ is an $S$-semimodule. If $M$ is a multiplication semimodule, then for every maximal ideal $\mathfrak{m}$ of $S$, either $M = \{x \in M \mid x =qx~\text{ for some}~ q\in \mathfrak{m}\}$ or $M$ is $\mathfrak{m}$-cyclic.
\end{lemma}
 
\begin{thm}\label{thmjm}
    Let $M$ be an $MC$ multiplication $S$-semimodule and $I$ a $1$-absorbing prime ideal of $S$. If $abx\in IM$, then $ab \in I$ or $x\in IM$ for all nonunits $a,b \in S$ and $x\in M$.
\end{thm}

\begin{proof}
    Let $a,b \in S$ be nonunits and let $x\in M$. If $x=0$, the result is obviously true. So assume that $x\not =0$. Suppose that $abx\in IM$ and $ab\notin I$. Let $J=\{s\in S\mid sx\in IM\}$. If $J\not= S$, then there is a maximal ideal $\mathfrak{m}$ of $S$ such that $J\subseteq \mathfrak{m}$. Since $M$ is $MC$, then $\{z \in M \mid z =qz~\text{ for some}~ q\in \mathfrak{m}\}=\{0\}$. Since $x\in M\setminus \{0\}$ and $M$ is a multiplication $S$-semimodule, then by Lemma \ref{lemjm}, $M$ is $\mathfrak{m}$-cyclic. So there exist $s\in S$, $p \in \mathfrak{m}$, and $y\in M$ such that $s+p=1$ and $sM \subseteq Sy$. It follows that $sx=ty$ for some $t\in S$ which implies $sabx=taby\in IM$ since $abx\in IM$. So we have $s^2abx\in sIM=IsM\subseteq ISy=Iy$. Hence $s^2abx=iy$ for some $i\in I$. But then $staby=s^2abx=iy$. Note that $y\not=0$, for if $y=0$, then $sx=ty=0$ and so $x=sx+px=px$ but $M$ is $MC$ implies $1=p\in \mathfrak{m}$, a contradiction. Again, since $M$ is $MC$, $stab=i\in I$. If $st$ is a unit in $S$, then $ab\in I$, a contradiction. If $st$ is a nonunit in $S$, then since $stab\in I$, $ab\notin I$, and $I$ is a $1$-absorbing prime ideal of $S$, we have $st\in I$. So $s^2x=sty\in IM$ and hence $s^2\in J\subseteq \mathfrak{m}$. Since $\mathfrak{m}$ is a prime ideal, $s\in \mathfrak{m}$. But then $1=s+p\in \mathfrak{m}$, a contradiction. Thus $J= S$ and so $1\in J$, that is, $x\in IM$. Therefore, $ab \in I$ or $x\in IM$.  
\end{proof}

\begin{thm}\label{thmIM}
    Let $M$ be an $MC$ multiplication $S$-semimodule and $I$ be a proper ideal of $S$. Then the following statements are equivalent:  
    \begin{enumerate}
        \item[(1)] $IM$ is a $1$-absorbing prime subsemimodule of $M$.
        \item[(2)] $I$ is a $1$-absorbing prime ideal of $S$.
    \end{enumerate}
    \end{thm}
\begin{proof} 
First, note that since $M$ is an $MC$ multiplication $S$-semimodule, then by \cite[Theorem 9]{NG}, $IM\not= M$. \\
$(1)\Rightarrow (2)$. Suppose that $abc\in I$ for some nonunits $a,b,c\in S$. Then $abcM\subseteq IM$. By (1) and Theorem \ref{th1}, either $ab \in (IM:_{S} M)$ or $cM \subseteq IM$. So $abM\subseteq IM$ or $cM\subseteq IM$. Then $SabM\subseteq IM$ or $ScM\subseteq IM$. It follows from \cite[Theorem 9]{NG} that $ab\in Sab\subseteq I$ or $c\in Sc\subseteq I$. Thus (2) holds.  \\
$(2)\Rightarrow (1)$. Suppose that $abx\in IM$ for some nonunits $a,b\in S$ and $x\in M$. Then by Theorem \ref{thmjm}, we have $ab \in I$ or $x\in IM$. But if $ab\in  I$, then $abM\subseteq IM$, that is, $ab\in (IM:_{S} M)$. Hence $ab\in (IM:_{S} M)$ or $x\in IM$. Therefore, (1) holds.
\end{proof}

Part (1) of the following corollary proves that the converse of Proposition \ref{p2} is true under the condition "$M$ is an $MC$ multiplication $S$-semimodule".

\begin{corollary}\label{cormc}
      Let $M$ be an $MC$ multiplication $S$-semimodule and $N$ be a proper subsemimodule of $M$. Then  
    \begin{enumerate}
        \item[(1)] $N$ is a $1$-absorbing prime subsemimodule of $M$ if and only if $(N :_S M)$ is a $1$-absorbing prime ideal of $S$.
        \item[(2)] $N$ is a $1$-absorbing prime subsemimodule of $M$ if and only if $N = IM$ for some $1$-absorbing prime ideal $I$ of $S$. 
    \end{enumerate} 
\end{corollary}

\begin{lemma}\label{lem0}
     Let $M_{1}$, $M_{2}$ be $S$-semimodules and $N_{1}$, $N_{2}$ be subsemimodules of $M_{1}$, $M_{2}$, respectively. Let $f:M_{1}\rightarrow M_{2}$ be a semimodule
homomorphism. 
\begin{enumerate}
    \item[(1)]  $(N_{2} :_{S} M_{2}) \subseteq (f^{-1}(N_{2}):_{S} M_{1})$.
    \item[(2)] If $f$ is onto, then $(N_{1} :_{S} M_{1}) \subseteq (f(N_{1}):_{S} M_{2})$.
\end{enumerate}
\end{lemma}
\begin{proof}
    \begin{enumerate}
        \item[(1)]  Let $s\in (N_{2} :_{S} M_{2})$ and let $m_{1}\in M_{1}$, then $sM_{2} \subseteq N_{2}$. Now $f(sm_{1})=sf(m_{1})\in sM_{2}\subseteq N_{2}$. So $sm_{1}\in f^{-1}(N_{2})$. Hence $sM_{1}\subseteq f^{-1}(N_{2})$. That is, $s\in (f^{-1}(N_{2}):_{S} M_{1})$. 
        \item[(2)] Let $s\in (N_{1}:_{S} M_{1})$ and let $m_{2}\in M_{2}$. So there exists $m_{1}\in M_{1}$ such that $m_{2}=f(m_{1})$ since $f$ is onto. But $sM_{1}\subseteq N_{1}$, so $sm_{1}\in N_{1}$. Hence $sm_{2}=f(sm_{1})\in f(N_{1})$. Thus $sM_{2}\subseteq f(N_{1})$. This means that $s\in (f(N_{1}):_{S} M_{2})$.
    \end{enumerate}
\end{proof}

Let $S$ be a semiring and $M$ be an $S$-semimodule. Recall that a proper subsemimodule $N$ of $M$ is called a strong subsemimodule if for each $x\in N$, there exists $y\in N$ such that $x+y=0$ \cite{MP}.

\begin{proposition}
    Let $M_{1},M_{2}$ be $S$-semimodules and $f:M_{1}\rightarrow M_{2}$ be a semimodule
homomorphism. 
\begin{enumerate}
    \item[(1)] If $N_{2}$ is a $1$-absorbing prime subsemimodule of $M_{2}$ and $\text{Im}(f)\nsubseteq N_2$, then $f^{-1}(N_{2})$ is a $1$-absorbing prime subsemimodule of $M_{1}$.
    \item[(2)] If $f$ is onto, $N_{1}$ is a $1$-absorbing prime subtractive strong subsemimodule of $M_{1}$, and $\text{Ker}(f)\subseteq N_{1}$, then $f(N_{1})$ is a $1$-absorbing prime subsemimodule of $M_{2}$.
\end{enumerate} 
\end{proposition} 

\begin{proof}
\begin{enumerate}
    \item[(1)] First, since $\text{Im}(f)\nsubseteq N_2$, $f^{-1}(N_{2})\not= M_1$. Suppose that $abm_{1} \in f^{-1}(N_{2})$, where $a,b$ are nonunits of $S$ and $m_{1} \in M_{1}$.
Then $abf(m_{1})=f(abm_1) \in N_{2}$. Since $N_{2}$ is a $1$-absorbing prime subsemimodule of $M_2$, we have $ab \in (N_{2} :_{S} M_{2})$ or $f(m_{1}) \in N_{2}$. But $(N_{2} :_{S} M_{2}) \subseteq (f^{-1}(N_{2}):_{S} M_{1})$ by Lemma \ref{lem0} part (1). It follows that $ab \in (f^{-1}(N_{2}):_{S} M_{1})$ or $m_{1} \in f^{-1}(N_{2})$.
Thus $f^{-1}(N_{2})$ is a $1$-absorbing prime subsemimodule of $M_{1}$.
\item[(2)] Suppose that $abm_{2} \in f(N_{1})$, where $a,b$ are nonunits of $S$ and $m_{2} \in M_{2}$. Then there exists $m_{1} \in M_{1}$ such that $f(m_{1}) = m_{2}$ since $f$ is onto. So $f(abm_{1})= abm_{2} \in f(N_{1})$. We claim that $abm_{1}\in N_{1}$. Since $f(abm_{1})\in f(N_{1})$, there exists $n_{1}\in N_{1}$ such that $f(abm_{1})=f(n_{1})$. Since $N_{1}$ is a strong subsemimodule, there exists $n_{1}'\in N_{1}$ such that $n_{1}+n_{1}'=0$. Then $f(abm_{1}+n_{1}')=0$. Hence $abm_{1}+n_{1}' \in \text{Ker}(f) \subseteq N_{1}$ but $N_{1}$ is subtractive and $n_{1}'\in N_{1}$, so $abm_{1}\in N_{1}$. Since $N_{1}$ is a $1$-absorbing prime subsemimodule of $M_{1}$, either $ab \in (N_{1} :_{S} M_{1})$ or $m_{1} \in N_{1}$. But $(N_{1} :_{S} M_{1}) \subseteq (f(N_{1}):_{S} M_{2})$ by Lemma \ref{lem0} part (2). So $ab \in ( f (N_{1}) :_{S} M_{2})$ or $m_{2}=f(m_{1}) \in f (N_{1})$. Thus $f(N_{1})$ is a $1$-absorbing prime subsemimodule of $M_{2}$.
\end{enumerate}  
\end{proof}

\begin{proposition}\label{prop1}
 Let $T$ be a nonempty multiplicatively closed subset of $S$, and $N$ a subsemimodule of $M$. If $T^{-1} N\not= T^{-1}M$ and $N$ is $1$-absorbing prime in $M$, then $T^{-1} N$ is $1$-absorbing prime in $T^{-1}M$.    
\end{proposition}

\begin{proof}
    Suppose that $\frac{a}{t}\frac{b}{t'}\frac{m}{t''}\in T^{-1} N$ for some nonunits $\frac{a}{t},\frac{b}{t'}\in T^{-1} S$ and $\frac{m}{t''}\in T^{-1} M$. So $\frac{abm}{tt't''}\in T^{-1} N$. This implies that $\frac{abm}{tt't''}=\frac{n}{s}$ for some $n\in N$ and $s\in T$. So there exists $v\in T$ such that $vsabm=vtt't''n$. It follows that $abvsm=vsabm\in N$. But since $\frac{a}{t}$ and $\frac{b}{t'}$ are nonunits in $T^{-1} S$, $a$ and $b$ are nonunits in $S$. Since $N$ is $1$-absorbing prime in $M$, we have $ab\in (N:_S M)$ or $vsm\in N$. Thus $\frac{a}{t}\frac{b}{t'}=\frac{ab}{tt'}\in T^{-1}(N:_S M)\subseteq (T^{-1}N:_{T^{-1}S} T^{-1}M)$ or $\frac{m}{t''}=\frac{vsm}{vst''}\in T^{-1} N$. Therefore, $T^{-1} N$ is $1$-absorbing prime in $T^{-1}M$. 
    \end{proof}

\section{Weakly $1$-absorbing prime subsemimodules}
We start this section by defining the concept of weakly $1$-absorbing prime subsemimodules over commutative semirings:
\begin{definition}\label{d2}
    Let $S$ be a commutative semiring and $M$ an $S$-semimodule. A proper subsemimodule $N$ of $M$ is called a weakly $1$-absorbing prime subsemimodule of $M$ if for all nonunits $a,b \in S$ and $m \in M$, $0\not= abm\in N$ implies 
$ab \in (N:_{S} M)$ or $m \in N$.
\end{definition} 

 Clearly, from Definitions \ref{d1} and \ref{d2}, every $1$-absorbing prime subsemimodule is weakly $1$-absorbing prime, but the converse is not true in general, as shown in the following example. 

\begin{example}\label{exwnot1}
    Let $S=\mathbb{Z}^{\circ}$, $M=\mathbb{Z}^{\circ}\times\mathbb{Z}_{8}$, and $N=\{(0,\overline{0})\}$. Then clearly, by definition, $N$ is a weakly $1$-absorbing prime subsemimodule of $M$. However, $N$ is not 1–absorbing prime subsemimodule of $M$ since $2\cdot 2\cdot (0,\overline{2})=(0,\overline{0})\in N$ but $2\cdot 2=4\notin (N:_{S} M)$ and $(0,\overline{2})\notin N$. 
\end{example}

\begin{lemma}\label{lem34}
     Let $M$ be an $S$-semimodule and $N$ a subsemimodule of $M$. If $N=N_{1}\cup N_{2}$, where $N_{1},N_{2}$ are subtractive subsemimodules of $M$, then $N=N_{1}$ or $N=N_{2}$. 
\end{lemma}
\begin{proof}
    Assume that $N\not= N_{1}$ and $N\not= N_{2}$. Take $x_{1}\in N\setminus N_{1}$ and $x_{2}\in N\setminus N_{2}$. Hence $x_{1}\in N_{2}$ and $x_{2}\in N_{1}$. Now $x_{1}+x_{2}\in N=N_{1}\cup N_{2}$. If $x_{1}+x_{2}\in N_{1}$, then since $x_{2}\in N_{1}$ and $N_{1}$ is subtractive, we have $x_{1}\in N_{1}$, a contradiction. Similarly, if $x_{1}+x_{2}\in N_{2}$, then since $x_{1}\in N_{2}$ and $N_{2}$ is subtractive, we have $x_{2}\in N_{2}$, a contradiction. Thus $N=N_{1}$ or $N=N_{2}$.
\end{proof}

The following theorem gives some characterizations of weakly $1$-absorbing prime subtractive subsemimodules.

\begin{thm}\label{th3} 
    Let $M$ be an $S$-semimodule and $N$ a proper subtractive subsemimodule of $M$. Then the following statements are equivalent:
    \begin{enumerate}
        \item[(1)] $N$ is a weakly $1$-absorbing prime subsemimodule of $M$.
  \item[(2)] For all nonunits $a,b\in S$ with $ab \notin (N:_{S} M)$, $(N :_{M} ab) = (0:_{M} ab) \cup N$.
  \item[(3)] For all nonunits $a,b\in S$ with $ab \notin (N:_{S} M)$, $(N :_{M} ab) = (0:_{M} ab)$ or $(N :_{M} ab) = N$.
  \item[(4)] For all nonunits $a,b\in S$ and subsemimodule $K$ of $M$ with $0\not=abK \subseteq N$, then $ab\in (N:_S M)$ or $K\subseteq N$.
  \item[(5)] For all nonunit $a\in S$, proper ideal $J$ of $S$, and subsemimodule $K$ of $M$ with $0\not=aJK \subseteq N$, either $aJ \subseteq (N:_{S} M)$ or $K \subseteq N$.
   \item[(6)] For all proper ideals $I,J$ of $S$, and subsemimodule $K$ of $M$ with $0\not=IJK \subseteq N$, either $IJ \subseteq (N:_{S} M)$ or $K \subseteq N$.
    \end{enumerate}
\end{thm} 

\begin{proof}
 (1) $\Rightarrow$ (2). Let $a,b\in S$ be nonunits such that $ab \notin (N:_{S} M)$.
Let $m \in (N :_{M} ab)$ and suppose $m\notin (0:_{M} ab)$. Then $0\not=abm \in N$. But $N$ is weakly $1$-absorbing prime subsemimodule and $ab \notin (N:_{S} M)$, so $m \in N$. Thus $(N :_{M} ab) \subseteq (0:_{M} ab) \cup N$. The reverse inclusion is always true.

(2) $\Rightarrow$ (3). This implication follows from Lemma \ref{lem34} and the fact that the subsemimodules $(0:_{M} ab)$ and $N$ are subtractive.

(3) $\Rightarrow$ (4). Assume that $0\not=abK \subseteq N$ for some nonunits $a,b\in S$ and some subsemimodule $K$ of $M$ and assume that $ab \notin (N:_{S} M)$. Since $abK \subseteq N$, $K \subseteq (N :_{M} ab)$. So by (3), $K\subseteq (0:_{M} ab)$ or $K\subseteq N$. But $K\nsubseteq (0:_{M} ab)$ since $abK\not = 0$. Thus $K\subseteq N$.

(4) $\Rightarrow$ (5). Suppose that $0\not=aJK \subseteq N$, where $a\in S$ is a nonunit, $J$ a proper ideal of $S$, and $K$ a subsemimodule of $M$. Suppose that $aJ \nsubseteq (N:_{S} M)$ and $K\nsubseteq N$. Then $ab\notin (N:_{S} M)$ for some $b\in J$. Since $aJK\not=0$, there exists $c\in J$ such that $acK\not=0$. Since $J$ is proper, $b,c,b+c$ are nonunits. We claim that $abK=0$. If $abK\not=0$, then $0\not=abK\subseteq aJK\subseteq N$. So by $(4)$ and since $ab\notin (N:_{S} M)$, we have $K\subseteq N$, a contradiction. Thus $abK=0$. Now $0\not= acK=a(b+c)K\subseteq aJK\subseteq N$. Again by (4) and since $K\nsubseteq N$, we have $a(b+c)\in (N:_S M)$. But since $0\not=acK\subseteq aJK\subseteq N$ and $K\nsubseteq N$, then $ac\in (N:_S M)$. But $a(b+c)\in (N:_S M)$ and $(N:_S M)$ is subtractive, so $ab\in (N:_S M)$, a contradiction. 

(5) $\Rightarrow$ (6). Suppose that $0\not=IJK \subseteq N$ for some proper ideals $I,J$ of $S$ and some subsemimodule $K$ of $M$. Suppose that $IJ \nsubseteq (N:_{S} M)$ and $K \nsubseteq N$. Then $aJ\nsubseteq (N:_{S} M)$ for some $a\in I$. If $aJK\not=0$, then $0\not=aJK\subseteq IJK\subseteq N$. So by (5) and since $aJ\nsubseteq (N:_{S} M)$, we have $K\subseteq N$, a contradiction. Thus $aJK=0$. Now since $IJK\not=0$, there exists $b\in I$ such that $bJK\not=0$. So $0\not=bJK\subseteq IJK\subseteq N$. By (5) and since $K\nsubseteq N$, we have $bJ\subseteq  (N:_{S} M)$. But $0\not=bJK=(a+b)JK\subseteq IJK\subseteq N$ and $K\nsubseteq N$, then again by (5), $(a+b)J\subseteq (N:_{S} M)$. Since $(N:_S M)$ is subtractive, $aJ\subseteq (N:_{S} M)$, a contradiction.
   
(6) $\Rightarrow$ (1). Assume that $0\not=abm \in N$ for some nonunits $a,b\in S$ and some $m\in M$. Take $I=Sa$, $J=Sb$, and $K = Sm$. Then $0\not=abm\in IJK=Sabm \subseteq N$. Thus by (6), $ab \in IJ\subseteq (N:_{S} M)$ or $m\in K\subseteq N$. Thus $N$ is a weakly $1$-absorbing prime subsemimodule of $M$.
\end{proof}	

 \begin{thm}
   Let $M$ be a multiplication $S$-semimodule and $N$ a proper subtractive subsemimodule of $M$. Then the following statements are equivalent:
   \begin{enumerate}
    \item[(1)] $N$ is a weakly $1$-absorbing prime subsemimodule of $M$.
  \item[(2)] For all subsemimodules $N_1,N_2,N_3$ of $M$ with $0\not=N_1N_2N_3\subseteq N$, either $N_1N_2 \subseteq N$ or $N_3 \subseteq N$.
    \end{enumerate}   
 \end{thm}

 \begin{proof}
    Follows from Theorem \ref{th3} and the proof of Theorem \ref{th2N}.
 \end{proof}

\begin{thm}\label{thmjm2}
    Let $M$ be an $MC$ multiplication $S$-semimodule and $I$ a weakly $1$-absorbing prime ideal of $S$. If $0\not=abx\in IM$, then $ab \in I$ or $x\in IM$ for all nonunits $a,b \in S$ and $x\in M$.
\end{thm}

\begin{proof}
    Let $a,b \in S$ be nonunits and $x\in M$. Suppose that $0\not=abx\in IM$ and $ab\notin I$. Let $J=\{s\in S\mid sx\in IM\}$. If $J\not= S$, then there is a maximal ideal $\mathfrak{m}$ of $S$ such that $J\subseteq \mathfrak{m}$. Since $M$ is an $MC$ multiplication $S$-semimodule, then as in the proof of Theorem \ref{thmjm}, $M$ is $\mathfrak{m}$-cyclic. So there exist $s\in S$, $p \in \mathfrak{m}$, and $y\in M$ such that $s+p=1$ and $sM \subseteq Sy$. It follows that $sx=ty$ for some $t\in S$. Again, by the proof of Theorem \ref{thmjm}, we have $y\not=0$ and $staby=iy$ for some $i\in I$. Since $M$ is $MC$, $stab=i\in I$. If $i=0$, then $staby=0$ implies $staby+ptaby=taby$ and so $ptaby=taby$. If $taby\not=0$, then since $M$ is $MC$, we have $p=1\in \mathfrak{m}$, a contradiction. So $taby=0$. But then $sabx=taby=0$. Since $abx\not=0$ and $M$ is $MC$, so $s=0$ and hence $p=1\in \mathfrak{m}$, again a contradiction. Thus $i\not=0$. So $0\not=i=stab\in I$. Since $I$ is a weakly $1$-absorbing prime ideal of $S$, then by a similar argument to the proof of Theorem \ref{thmjm}, we obtain $x\in IM$.  
\end{proof}

\begin{thm}\label{thmIM2}
    Let $M$ be an $MC$ multiplication $S$-semimodule and $I$ be a proper ideal of $S$ such that $IM$ is subtractive. Then the following statements are equivalent:  
    \begin{enumerate}
        \item[(1)] $IM$ is a weakly $1$-absorbing prime subsemimodule of $M$.
        \item[(2)] $I$ is a weakly $1$-absorbing prime ideal of $S$.
    \end{enumerate}
    \end{thm}
\begin{proof} 
First, since $M$ is an $MC$ multiplication $S$-semimodule, then by \cite[Theorem 9]{NG}, $IM\not= M$. \\
$(1)\Rightarrow (2)$. Suppose that $0\not=abc\in I$ for some nonunits $a,b,c\in S$. Since $M$ is $MC$, $M$ is faithful \cite{NG}. So $\text{Ann}(M)=0$. If $abcM=0$, then $abc\in \text{Ann}(M)=0$, a contradiction. So $0\not=abcM\subseteq IM$. Since $IM$ is subtractive, then by (1) and Theorem \ref{th3}, $ab \in (IM:_{S} M)$ or $cM \subseteq IM$. So $abM\subseteq IM$ or $cM\subseteq IM$. Then by \cite[Theorem 9]{NG}, we have $ab\in I$ or $c\in I$. Therefore, (2) holds.  \\
$(2)\Rightarrow (1)$. Suppose that $0\not=abx\in IM$ for some nonunits $a,b\in S$ and $x\in M$. By Theorem \ref{thmjm2}, we have $ab \in I$ or $x\in IM$. So $abM\subseteq IM$ or $x\in IM$. Hence $ab\in (IM:_{S} M)$ or $x\in IM$. Therefore, (1) holds.
\end{proof}

\begin{corollary}\label{cormc2}
      Let $M$ be an $MC$ multiplication $S$-semimodule and $N$ be a proper subtractive subsemimodule of $M$. Then  
    \begin{enumerate}
        \item[(1)] $N$ is a weakly $1$-absorbing prime subsemimodule of $M$ if and only if $(N :_S M)$ is a weakly $1$-absorbing prime ideal of $S$.
        \item[(2)] $N$ is a weakly $1$-absorbing prime subsemimodule of $M$ if and only if $N = IM$ for some weakly $1$-absorbing prime ideal $I$ of $S$. 
    \end{enumerate} 
\end{corollary}

\begin{proposition}
 Let $T$ be a nonempty multiplicatively closed subset of $S$, and $N$ a subsemimodule of $M$. If $T^{-1} N\not= T^{-1}M$ and $N$ is a weakly $1$-absorbing prime in $M$, then $T^{-1} N$ is a weakly $1$-absorbing prime in $T^{-1}M$.    
\end{proposition}

\begin{proof}
    Suppose that $0\not=\frac{a}{t}\frac{b}{t'}\frac{m}{t''}\in T^{-1} N$ for some nonunits $\frac{a}{t},\frac{b}{t'}\in T^{-1} S$ and $\frac{m}{t''}\in T^{-1} M$. So $0\not=\frac{abm}{tt't''}\in T^{-1} N$. As in the proof of Proposition \ref{prop1}, $uabm\in N$ for some $u\in T$. Since $\frac{abm}{tt't''}\not=0$, then $uabm\not=0$. So we have $0\not=uabm\in N$. But since $\frac{a}{t}$ and $\frac{b}{t'}$ are nonunits in $T^{-1} S$, $a$ and $b$ are nonunits in $S$. By hypothesis, $ab\in (N:_S M)$ or $um\in N$. Thus $\frac{a}{t}\frac{b}{t'}=\frac{ab}{tt'}\in T^{-1}(N:_S M)\subseteq (T^{-1}N:_{T^{-1}S} T^{-1}M)$ or $\frac{m}{t''}=\frac{um}{ut''}\in T^{-1} N$. Therefore, $T^{-1} N$ is a weakly $1$-absorbing prime in $T^{-1}M$. 
    \end{proof}

\begin{proposition}
  Let $M$ be an $S$-semimodule and $x \in M$. Suppose that $(0:_S x) \subseteq (Sx:_{S} M)$ and $Sx$ is subtractive. Then $Sx$ is a weakly $1$-absorbing prime subsemimodule of $M$ if and only if it is $1$-absorbing prime.
\end{proposition} 
\begin{proof}
 Suppose that $Sx$ is a weakly $1$-absorbing prime subsemimodule of $M$ and suppose $abm\in Sx$ for some nonunits $a,b\in S$ and some $m\in M$. If $abm\not=0$, then either $ab \in (Sx:_{S} M)$ or $m \in Sx$ since $Sx$ is weakly $1$-absorbing prime. If $abm=0$, then $ab(x + m) =abx\in Sx$. Either $ab(x + m) \not= 0$ or $ab(x + m) = 0$. If $ab(x + m) \not= 0$, then since $Sx$ is weakly $1$-absorbing prime, we have either $ab \in (Sx:_{S} M)$ or $x+m \in Sx$. But $Sx$ is subtractive, so $ab \in (Sx:_{S} M)$ or $m \in Sx$. If $ab(x + m)=0$, then $abx = 0$ (since $abm=0$). So $ab \in (0:_S x)\subseteq (Sx :_{S} M)$. Therefore, $Sx$ is $1$-absorbing prime. The converse is clear. 
\end{proof}

\begin{definition}
  Let $M$ be an $S$-semimodule and $N$ a weakly $1$-absorbing prime subsemimodule of $M$. Let $a, b \in S$ be nonunits and $m\in M$. We say that $(a, b, m)$ is a triple-zero of $N$ if $abm = 0$, $ab \notin (N:_{S} M)$, and $m \notin N$.
\end{definition} 

\begin{example}
   Let $S=\mathbb{Z}^{\circ}$, $M=\mathbb{Z}_{20}$, and $N=\{\overline{0}\}$. Then  $(2,2,\overline{5})$ is a triple-zero of $N$ since $2\cdot 2\cdot \overline{5}=\overline{0}$, $2\cdot 2=4\notin (N:_{S} M)$, and $\overline{5} \notin N$. 
\end{example}

\begin{remark}\label{rem2}
    If $N$ is a weakly $1$-absorbing prime subsemimodule of $M$ that is not $1$-absorbing prime, then $N$ has a triple-zero $(a, b, m)$ for some nonunits $a, b \in S$ and $m\in M$.
\end{remark} 

\begin{thm}\label{th2}
    Let $S$ be a local semiring and $M$ an $S$-semimodule. If $N$ is a weakly $1$-absorbing prime subtractive subsemimodule of
$M$ and $(a, b, m)$ is a triple-zero of $N$ for some nonunits $a, b\in S$ and $m \in M$, then
\begin{enumerate}
 \item[(1)] $abN= a(N :_{S } M)m = b(N :_{S } M)m = 0$. 
 \item[(2)]  $a(N:_{S } M)N = b(N:_{S } M)N = (N :_{S } M)^{2}m =0$.
\end{enumerate}
\end{thm} 

\begin{proof}
First, note that since $(a, b, m)$ is a triple-zero of $L$, then $abm = 0$, $ab \notin (N:_{S} M)$, and $m \notin N$.
\begin{enumerate}

\item[(1)] Suppose that $abN\not = 0$. So there exists $n \in N$ such that $abn\not=0$. Now, since $abm =0$, we have $ab(m+n) = abm+ abn= abn \not=0$. So $0\not= ab(m + n) \in N$ but since $N$ is a weakly $1$-absorbing prime subsemimodule, so $ab \in (N :_{S } M)$ or $m + n \in N$. But $N$ is subtractive, so $ab \in (N :_{S } M)$ or $m \in N$, a contradiction since $(a, b, m)$ is a triple-zero of $N$. Hence $abN =0$. Next, assume that $a(N:_{S} M)m \not= 0$. Then there exists $t \in (N:_{S } M)$ such that $atm \not= 0$. So $a(t + b)m=atm+abm=atm+0=atm\not=0$. So $0\not=a(t+b)m=atm\in atM\subseteq aN\subseteq N$. Since $S$ is a local semiring and $t,b$ are nonunits, then $t + b$ is a nonunit. But $N$ is a weakly $1$-absorbing prime subsemimodule, so $a(t + b) \in (N :_{S } M)$ or $m \in N$. Since $N$ is subtractive, then by \cite[Lemma 3 (i)]{ES}, $(N:_{S } M)$ is subtractive but $at \in (N:_{S } M)$ and $at +ab \in (N:_{S } M)$, so $ab \in (N :_{S } M)$. Hence $ab \in (N :_{S } M)$ or $m \in N$, a contradiction since $(a, b, m)$ is a triple-zero of $N$. Thus $a(L  :_{S } M)m = 0$. Similarly, we have $b(L  :_{S } M)m = 0$. 
\item[(2)] Suppose that $a(N:_{S } M)N\not=0$. Then there exist $t \in (N:_{S } M)$ and $n \in N$ such that $atn \not= 0$. By (1), $abN= a(N:_{S } M)m = 0$, so $abn=atm=0$. Thus $a(b + t)(m + n) = abm + abn + atm + atn = atn \not= 0$. We have $0\not=a (b + t)(m + n)=atn\in N$. But $b + t$ is a nonunit (since $S$ is local) and $N$ is weakly $1$-absorbing prime, so $a(b + t) \in (N:_{S } M)$ or $m + n \in N$. But $N$ is subtractive implies $ab \in (N:_{S } M)$ or $m \in N$, a contradiction since $(a, b, m)$ is a triple-zero of $N$. Thus $a(N:_{S } M)N=0$. Similarly, we have $b(N:_{S } M)N=0$. Finally, we show that $(N:_{S } M)^{2}m =0$. Suppose that $(N:_{S } M)^{2}m \not= 0$. Then there exist $s,t \in (N:_{S } M)$ such that $stm \not= 0$. By above, we have $(a +s)(b +t)m = stm \not= 0$ and $0\not= (a + s)(b + t)m \in N$. Since $S$ is a local semiring, we have $a + s$ and $b + t$ are nonunits. Hence $(a + s)(b + t) \in (N:_{S } M)$ or $m \in N$. Since $(N:_{S } M)$ is subtractive, we have $ab \in (N :_{S } M)$ or $m \in N$, a contradiction since $(a, b, m)$ is a triple-zero of $N$. Therefore, $(N:_{S } M)^{2}m = 0$. 
\end{enumerate}
 \end{proof}

\begin{thm}\label{p4}
 Let $S$ be a local semiring and $M$ an $S$-semimodule. If $N$ is a weakly $1$-absorbing prime subtractive subsemimodule of $M$ that is not $1$-absorbing prime, then $(N:_{S } M)^{2}N= 0$.
\end{thm}
\begin{proof}
Assume that $N$ is a weakly $1$-absorbing prime subtractive subsemimodule of $M$ that is not $1$-absorbing prime. So by Remark \ref{rem2}, $N$ has a triple-zero $(a, b, m)$ for some nonunits $a, b \in S$ and $m \in M$. Assume that $(N:_{S } M)^{2}N\not=0$. Then $stn \not= 0$ for some $s,t \in (N:_{S } M)$ and $n \in N$. From Theorem \ref{th2}, we have $(a + s)(b + t)(m + n) = stn \not= 0$. So $0 \not=(a + s)(b + t)(m + n)=stn \in N$. But since $N$ is weakly $1$-absorbing prime, then $(a + s)(b + t) \in (N:_{S } M)$ or $ m+n \in N$ but $N$ is subtractive implies $ab \in (N:_{S } M)$ or $m \in N$, a contradiction since $(a, b, m)$ is a triple-zero of $N$. Thus $(N :_{S } M)^{2}N = 0$. 
\end{proof}

\begin{corollary}
     Let $S$ be a local semiring and $M$ an $S$-semimodule. If $N$ is a weakly $1$-absorbing prime subtractive subsemimodule of $M$ that is not $1$-absorbing prime, then $(N:_{S } M)^{3} \subseteq \text{Ann}(M)$. 
\end{corollary}

\begin{proof}
    Since $(N:_{S } M)M\subseteq N$, then $(N :_{S } M)^{3}M=(N:_{S } M)^{2}(N :_{S } M)M\subseteq (N :_{S } M)^{2}N$. But $(N:_{S } M)^{2}N= 0$ by Theorem \ref{p4}. So $(N :_{S } M)^{3}M=0$. Thus $(N :_{S } M)^{3} \subseteq \text{Ann}(M)$.
\end{proof}

\begin{corollary}\label{prop312} Let $S$ be a local semiring and $M$ be a multiplication $S$-semimodule. If $N$ is a weakly $1$-absorbing prime subtractive subsemimodule of $M$ that is not $1$-absorbing prime, then $N^{3} = 0$.
\end{corollary}
\begin{proof} Since $M$ is a multiplication $S$-semimodule, $N=(N:_{S } M)M $. Thus
\begin{align*}
    N^{3}=& ~(N:_{S } M)^{3}M\\
         =&~(N:_{S } M)^{2}(N:_{S } M)M\\
       =&~(N:_{S } M)^{2}N \\
     =&~0 \hspace{2cm} \text{by Theorem \ref{p4}}
\end{align*}

\end{proof}

\begin{thm}
    Let $S$ be a local semiring and $M$ an $MC$ multiplication $S$-semimodule. Let $I$ be a proper ideal of $S$ such that $I^3\not=0$ and $IM$ is subtractive. Then the following statements are equivalent:  
    \begin{enumerate}
        \item[(1)] $IM$ is a weakly $1$-absorbing prime subsemimodule of $M$.
          \item[(2)] $IM$ is a $1$-absorbing prime subsemimodule of $M$.
        \item[(3)] $I$ is a $1$-absorbing prime ideal of $S$.
        \item[(4)] $I$ is a weakly $1$-absorbing prime ideal of $S$.
    \end{enumerate}
    \end{thm}
\begin{proof} 
$(1)\Rightarrow (2)$. Since $I^3\not=0$, and $M$ is faithful (since $M$ is $MC$), then $I^3M\not=0$ and so $(IM)^3\not=0$. So by (1) and Corollary \ref{prop312}, (2) holds.  \\
$(2)\Rightarrow (3)$. Follows from Theorem \ref{thmIM}.\\
$(3)\Rightarrow (4)$. Clear.\\
$(4)\Rightarrow (1)$. Follows from Theorem \ref{thmIM2}.

\end{proof}


\vspace{0.2cm}


\begin{thebibliography}{99}

\normalsize
\baselineskip=17pt
\bibitem{MAMS} M. Adarbeh and M. Saleh, \emph{On $2$-absorbing ideals of noncommutative semirings}, J. Algebra and Its Appl., 2024, to appear. DOI: 10.1142/S0219498825502780
\bibitem{AS} D. F. Anderson and E. Smith, \emph{Weakly prime ideals}, Houst. J. Math., 29(4) (2003), 831–840.
\bibitem{REA} R. E. Atani, \emph{The ideal theory in quotients of commutative semirings}, Glas. Math., 42(62) (2007), 301-308.
\bibitem{ES}  R. E. Atani and S. E. Atani, \emph{On subsemimodules of semimodules}, Buletinul Academiei De
Stiinte a Republicii Moldova Mathematica, 2(63) (2010), 20–30. 
\bibitem{E} R. E. Atani, \emph{On prime subsemimodules of semimodules}, Int. J. of Algebra, 4(26), 1299–1306
\bibitem{Ab_{2}} A. Badawi, \emph{On $2$-absorbing ideals in commutative rings}, Bull. Aust. Math. Soc., 75 (2007) 417–429.
\bibitem{AB3} A. Badawi and D. A. Youseﬁan, \emph{On weakly $2$-absorbing ideals of commutative rings}, Houst. J. Math., 39 (2013) 441–452.
 \bibitem{BY} A. Badawi, E. Yetkin Celikel, \emph{On $1$-absorbing primary ideals of commutative rings},
J. Algebra Appl. 19 (6) (2020) 2050111, 12 pp.
\bibitem{BYC} A. Badawi, E. Yetkin Celikel, \emph{On weakly $1$-absorbing primary ideals
of commutative rings}, Algebra Colloquium, 29 : 2 (2022), 189–202.
\bibitem{BN} H. Behzadipour and P. Nasehpour, \emph{On $2$-absorbing ideals of commutative semirings}, J. Algebra and Its Appl., 19(02)
(2020) Article ID:  2050034, 12 pp.
\bibitem{D} R. P. Deore, \emph{Characterizations of semimodules}. Southeast Asian Bull. Math., 36, 187–196 (2012).

\bibitem{MP} M. K. Dubey and P. Sarohe, \emph{On prime and primary avoidance theorem
for subsemimodules}, Quasigroups and Related Systems, 25 (2017), 221-229.
\bibitem{G}  J. S. Golan, \emph{Semirings and their applications}, Kluwer Academic Publishers, Dordrecht, 1999. 

\bibitem{KTY} S. Koç, U. Tekir, and E. Yıldız, \emph{On weakly $1$-absorbing prime ideals}. Ricerche di Matematica, (2021). 
\bibitem{PNV} P. Nasehpour, \emph{Valuation semirings}, Journal of Algebra and Its Applications, 16 (11) (2018), 1850073.
\bibitem{PNS} P. Nasehpour, \emph{Some remarks on semirings and their ideals}, Asian-Eur. J. Math., 12 (2020), Article ID: 2050002 (14 pages).
\bibitem{NG} R. R. Nazari and S. Ghalandarzadeh, \emph{Multiplication semimodules}, Acta Univ. Sapientiae, Mathematica, 11, 1 (2019) 172–185
\bibitem{PB}   S. H. Payrovi and S. Babaei, \emph{On the 2-absorbing ideals in commutative rings}, Bull. Malays. Math. Soc. 23,
 1511–1526 (2013)
\bibitem{MSIM} M. Saleh and I. Murra, \emph{On weakly $1$-absorbing primary ideals of commutative semirings}, Communications in Advanced Mathematical Sciences, (2022), 199-208.
\bibitem{U} A. E. Ugurlu, \emph{On $1$-absorbing prime submodules}, Moroccan Journal of Algebra
 and Geometry with Applications, 2(1) (2023), 162-173.
\bibitem{V} H. S. Vandiver, \emph{Note on a simple type of algebra in which the cancellation
law of addition does not hold.}, Bull. Amer. Math. Soc., 40(3) (1934), 916-920.
\bibitem{NNY} A. Yassine, M.J. Nikmehra, and R. Nikandish, \emph{On $1$-absorbing prime ideals of commutative rings}, J. Algebra Appl., 20(10) (2021), Paper No. 2150175. doi: 10.1142/S0219498821501759

\end{thebibliography}
\end{document}